\documentclass[11pt,leqno]{amsart}
\usepackage{enumerate}
\usepackage{amssymb,commath,amscd}
\usepackage{stix}
\usepackage{amsthm}
\usepackage{
	nameref,
}
\usepackage[colorlinks,linkcolor=black,anchorcolor=blue,citecolor=green]{hyperref}
\newtheorem{theorem}{Theorem}[section]
\newtheorem{lemma}[theorem]{Lemma}

\newtheorem{proposition}[theorem]{Proposition}

\theoremstyle{remark}
\newtheorem{remark}[theorem]{Remark}
\newtheorem{example}[theorem]{Example}

\numberwithin{equation}{theorem}

\allowdisplaybreaks

\newcommand{\injrad}{\operatorname{inj.rad}}
\newcommand{\conjrad}{\operatorname{conj.rad}}
\newcommand{\Ric}{\operatorname{Ric}}
\newcommand{\diam}{\operatorname{diam}}
\newcommand{\sndf}{{I\!I}}

\begin{document}
	\title[Stability of pure Nilpotent Structures] {Stability of pure Nilpotent Structures on collapsed Manifolds}
	
	
	\author{Zuohai Jiang}
	\address{School of Mathematical Sciences, Capital Normal Universiy, Beijing China}
	\email{jiangzuohai08@163.com}
	
	
	\author{Shicheng Xu}
	\address{School of Mathematical Sciences, Capital Normal Universiy, Beijing China}
	\curraddr{}
	\email{shichengxu@gmail.com}

	
	\date{\today}

	\begin{abstract} The goal of this paper is to study the stability of pure nilpotent structures on a manifold associated to different collapsed metrics. We prove that if two metrics on a $n$-manifold of bounded sectional curvature are $L_0$-bi-Lipchitz equivalent and sufficient collapsed (depending on $L_0$ and $n$), then up to a diffeomorphism, the underlying nilpotent Killing structures coincide with each other or one is embedded into another as a subsheaf. It improves Cheeger-Fukaya-Gromov's locally compatibility of pure nilpotent Killing structures for one collapsed metric of bounded sectional curvature to two Lipschitz equivalent metrics. As an application, we prove that those pure nilpotent Killing structures constructed by various smoothing method to a Lipschitz equivalent metric of bounded sectional curvature are uniquely determined by the original metric modulo a diffeomorphism.
	\end{abstract}
	
	\maketitle
	
	\setcounter{section}{-1}
	\section{Introduction}
	\subsection{Background}
	Let $(M^n,g)$ be a complete Riemannian $n$-manifold whose sectional curvature $|\sec(M,g)|\le 1$. It is called \emph{$\epsilon$-collapsed}, if for any $x\in M$ the injectivity radius satisfy $\injrad_g(x)\le \epsilon$. 
	Collapsed manifolds under bounded sectional curvature are extensively studied by Gromov, Cheeger-Gromov and Fukaya (\cite{Gromov1978}, \cite{Fukaya1987Collapsing,Fukaya1989collapsing},  \cite{CheegerGromov1990,CheegerGromov1990II}, \cite{CFG1992}). Since then many applications on manifolds of bounded sectional curvature were obtained (e.g., \cite{FangRong1999}, \cite{PRT1999}, \cite{PT1999}, etc. and survey papers \cite{Fu2006}, \cite{Rong2007}).

A \emph{maximally} collapsed manifold $(M^n,g)$, whose diameter and sectional curvature satisfy $\diam(M,g)\cdot |\sec(M,g)|^{1/2}< \epsilon(n)$, a constant depending on $n$, is characterized by Gromov's almost flat manifolds (\cite{Gromov1978,Ruh1982}), such that $M$ is a \emph{infra-nilmanifold} $\Gamma\backslash N$, where $N$ is a nilpotent Lie group and $\Gamma$ is a discrete affine transformation subgroup with a universal bounded index $[\Gamma:\Gamma\cap N]<w(n)$, a constant depending only on $n$. By a parametrized version of Gromov's almost-flat manifold, which is called Fukaya's fibration theorem \cite{Fukaya1987Collapsing,Fukaya1989collapsing}, the collapsing of a Riemannian manifold $(M,g)$ to a lower-dimensional manifold $(Y,h)$ corresponds to an affine fiber bundle $(M,Y,f)$ whose fiber is an almost flat manifold containing all collapsed directions.
A fiber bundle $(M,Y,f)$ with fiber a infra-nilmanifold $\Gamma\backslash N$, is called \emph{affine} if its structure group is contained in the affine transformation group of $\Gamma\backslash N$.

 In general an $\epsilon$-collapsed manifold $(M,g)$ are characterized by 
a \emph{nilpotent Killing structure} $\mathfrak n$, i.e., a sheaf of nilpotent Lie algebras of local vector fields pointing all $\epsilon$-collapsed directions, which are Killing fields with respect to a nearby metric $g_\epsilon$, and generate an action of a nilpotent Lie group on a normal cover of some neighborhood around points in $M$ (see \cite{CFG1992} or section 2.2).  A collapsed manifold $M$ decomposes along $\mathfrak n$ into ``orbits'', which are infra-nilmanifolds maybe of different dimension, tangent to the stalks of $\mathfrak n$ and absorbing all $\epsilon$-collapsed directions. The \emph{rank} of $\mathfrak n$ is defined to be minimal dimension of its orbits. A nilpotent Killing structure $\mathfrak n$ is called \emph{pure}, if the dimension of its stalk is locally constant.  For a pure structure $\mathfrak n$, we define its \emph{dimension} to be that of its stalk. Since collapsing can take place simultaneously on several length scales,  nilpotent Killing structures on a fixed $\epsilon$-collapsed metric also depend on the choice of $\epsilon$, the scale that one inspects. An $F$-structure constructed in \cite{CheegerGromov1990II} is an local action of a sheaf of tori (complete on a finite norm cover) which corresponds to the smallest length scale of collapsing, i.e., the injectivity radius at each point.

The above theorems fails for manifolds of bounded Ricci curvatrure; see \cite{Anderson1992}. However, it is well known that provided some additional conditions, such as a positive lower bound on conjugate radius, the manifolds of (lower) bounded Ricci curvature 
can be smoothed via various methods (e.g. \cite{DWY1996,PWY1999,Li2010}, etc.) to a nearby metric $g_\epsilon$ of bounded sectional curvature. If $(M,g)$ is \emph{$\epsilon$-volume collapsed}, i.e., the volume of every $1$-ball in $(M,g)$ is $\le \epsilon$, then the injectivity radius of the new metric $(M,g_\epsilon)$ could be also small (e.g. \cite{DWY1996}). Therefore a nilpotent Killing structure still exists on such manifolds. Via Ricci flow method (\cite{Hamilton1982}), nilpotent Killing structures are also known to exist on closed manifolds of lower bounded Ricci curvature, whose universal cover satisfies the $(\delta,r)$-Reifenberg condition (see \cite{HKRX}). In particular, Gromov's almost flat manifold theorem and Fukaya's fibration theorem holds \cite{DWY1996,PWY1999,HKRX}. It is a natural question that, besides the original metric, whether those nilpotent Killing structures substantially depend on those different smoothing methods? (see Remark \ref{rem-smooth} below)

The compatibility (resp. stability) of locally constructed \emph{pure} nilpotent Killing structures (resp. $F$-structures) around neighboring points for \emph{one fixed metric}, such that one fits inside another, is one of the key steps in construction of a global nilpotent Killing structure (resp. $F$-structure); see \cite{CFG1992} and \cite{CheegerGromov1990II}. The stability of pure nilpotent Killing structures associated to a continuous family of metrics $g(t)$ (sufficiently collapsed to a fixed limit) played an important role in solving a continuous version of  Klingenberg-Sakai conjecture (\cite{PRT1999}).

This paper is devoted to study the stability of pure nilpotent Killing structures for \emph{different collapsed metrics} on a fixed smooth manifold, and the uniqueness of nilpotent Killing structure obtained by various smoothing methods. 

\subsection{Main results}
Let $d_L$ be the \emph{Lipschitz distance} between two Riemannian metrics on $M$ defined by
$$d_L(g_1,g_2)=\inf_{\varphi:M\to M} \ln\left(\max \left\{\operatorname{dil}(\varphi), \operatorname{dil}(\varphi^{-1})\right\}\right),$$
where the infimum is taken over all diffeomorphisms $\varphi:M\to M$, and $\operatorname{dil}(\varphi)$ is the optimal Lipschitz constant of $\varphi$ (also called dilatation, cf. \cite{BBI2001}).
The main result of this paper is as follows.
\begin{theorem}\label{thm-lip-stable}
	Given $L_0>0, n\ge 1$, there is $\epsilon=\epsilon(L_0,n)>0$ such that the following holds.
	
	Let $g_i$ $(i=1,2)$ be two $\epsilon$-collapsed complete Riemannian metrics on a $n$-manifold $M^n$ whose sectional curvature $|\sec(M,g_i)|\le 1$. If the nilpotent Killing structure $\mathfrak n_i$ associated to $g_i$ is pure, and 
	\begin{equation}\label{cond-Lip-equiv}
	d_L(g_1,g_2)\le L_0,
	\end{equation}
	then there is a diffeomorphism $\Phi:M\to M$
	such that up to a permutation the push-forward $\Phi_*\mathfrak n_1$ is a subsheaf of $\mathfrak n_2$.  If the dimensions of $\mathfrak n_i$ are the same, then they are 
	isomorphic by $\Phi_*$ as sheaves.
\end{theorem}

Theorem \ref{thm-lip-stable} improves the local compatibility in \cite{CFG1992} and stability of pure nilpotent Killing structures in \cite{PRT1999} to a situation that two metrics $g_1$ and $g_2$ are not close in the Gromov-Hausdorff distance. 

Since there are infinitely many pairwise non-conjugate isometric $S^1$-actions on $S^2\times S^3$ (see Example \ref{exam-ty} below), the collapsing condition in Theorem \ref{thm-lip-stable} is essential. 

\begin{remark}\label{rem-main}
\item (\ref{rem-main}.1) 
Roughly speaking, the nilpotent structures are determined by the local fundamental group at the same length scale. Condition (\ref{cond-Lip-equiv}) is to guarantee the collapsing rates of $g_i$ are on a comparable level, such that the subgroups of $\pi_1(B_1(p))$ generated by $g_i$-small geodesic loops coincides on each collapsing length scale.

\item (\ref{rem-main}.2)
Our construction of $\Phi$ provides a slightly different way from \cite{CFG1992} (see Remark \ref{rem-isotopy}) to prove the local compatibility of pure nilpotent Killing structures on orthogonal fame bundles. By the method in this paper, we are able to construct a global nilpotent structure on an open manifold that only admits a local smoothing (e.g. \cite{Li2010}); see \cite{HKRX1}.
\end{remark}

As a corollary of Theorem \ref{thm-lip-stable}, we obtain the following uniqueness of pure nilpotent Killing structure by various smoothing techniques.
\begin{theorem}\label{thm-unique-nstr}
	Given any $n\ge 1, C>0$, there is a positive constant $\epsilon(n,C)>0$ such that for $0<\epsilon<\epsilon(n,C)$, pure nilpotent Killing structures on an $\epsilon$-volume collapsed $n$-manifold $(M^n,g)$  derived by different smoothing methods to a metric $g_0$ such that \begin{equation}\label{cond-lipnearby}
	d_L(g_0,g)\le 1 \text{ and } |\sec(M,g_0)|\le C
	\end{equation} are unique up to a diffeomorphism.
\end{theorem}

Theorem \ref{thm-unique-nstr} covers the nilpotent Killing structures constructed via Ricci flow in \cite{DWY1996,HKRX} and related evolutions such as \cite{Li2010} on manifolds of bounded Ricci curvature, or via embedding into a Hilbert space in \cite{PWY1999} on manifolds of lower bounded Ricci curvature under some additional regularity assumptions (e.g., a positive bound on conjugate radius). 

By Theorem \ref{thm-unique-nstr} again, Theorem \ref{thm-lip-stable} holds for any $\epsilon$-volume collapsed manifold $(M,g)$ such that there is a nearby metric $g_0$ satisfying (\ref{cond-lipnearby}). For example, if sectional curvature bound in Theorem \ref{thm-lip-stable} is replaced by 
\begin{equation*}
\Ric(M,g_i)\ge -(n-1), \quad \conjrad(M,g_i)\ge r_0
\end{equation*}
and $(M,g_i)$ is $\epsilon_0(L_0,n,r_0)$-volume collapsed,
then the nilpotent Killing structure exists \cite{PWY1999} and uniquely determined by $g_i$, and the stability result in Theorem 0.1 still holds for such $g_i$.

In practice, condition (\ref{cond-Lip-equiv}) naturally arises in Theorem \ref{thm-unique-nstr}, due to that different methods usually give rise to different curvature bounds. 

Indeed, let $C_i(t)$ $(i=1,2; t\ge 0)$ be universal sectional curvature bounds for two methods $\mathcal{S}_{i,t}$ respectively, i.e., 
$C_i(t)=\sup_{g} |\sec(M,g_{i,t})|$, where the supremum is taken over all smoothable metric $g$, and $g_{i,t}=\mathcal{S}_{i,t}(g)$ $(i=1,2)$ are smoothed metrics satisfying $d_L(g_{i,t},g)\le t$. Then after normalizing to $|\sec|\le 1$, the rescaled metrics will admit a Lipschitz control \begin{equation}\label{rescaled-control}
d_{L}(C_1g_{1,t}, C_2g_{2,t})\le \max\{\ln C_1/C_2,\ln C_2/C_1\}+t, \;\text{where $C_i=C_i(t)$.}
\end{equation}

\begin{remark}\label{rem-smooth}
	We point out the main issues on stability of nilpotent structures constructed by smoothing methods.
	
	First, the general sectional curvature bound $C_i(t)$ of $g_{i,t}$ would blow up at a different rate, as $g_{i,t}$ approaches $g$. For an $\epsilon$-volumed collapsed metric $g$, if the ratio $C_1(t)/C_2(t)$ is large as $t$ small, then the renormalized metrics $C_i(t)g_{i,t}$ maybe collapse on different scales. 
	
	Secondly, as the underlying metric continuously varies to different scales, sub-nilpotent structures may appear or vanish several times (see Example \ref{exam-ty} below). A coherence between nilpotent Killing structures of $g_{1,t}$ and $g_{2,t'}$ ($t<t'$) does not imply the same between $g_{1,t'}$ and $g_{2,t'}$, as $C_i(t)/C_i(t')$ is relative large. 
	
	Due to the two issues above, a uniform curvature bound is required in Theorem \ref{thm-unique-nstr}.
	
	For the same reason, the uniqueness in Theorem \ref{thm-unique-nstr} cannot follow from  previous results in \cite{CFG1992} or \cite{PRT1999}, where both of them essentially deal with two $C^{1,\alpha}$-close metrics with uniformly bounded curvature.
\end{remark}

Here is a typical example that carries different nilpotent Killing structures  (cf. \cite{PT1999}, \cite{wang1990einstein}).
\begin{example}\label{exam-ty}
Let $S^2\times S^3$ be endowed which the canonical product metric $h_0$. For a positive integer $k$, let $\rho_k$ be the isometric free $S^1$-action on $S^2\times S^3$ by $$e^{\theta\sqrt{-1}}\cdot (x,z;v,w)=(x,e^{k\theta\sqrt{-1}}z;e^{\theta\sqrt{-1}}v,e^{\theta\sqrt{-1}}w),$$ where $x$ is real, $z,v,w$ are complex coordinates. By \cite[Proposition 5.1]{OU72}, $\rho_k$ gives rise to circle bundles of distinct Euler class for different $k$. Thus $\rho_j$ and $\rho_k$ ($j\neq k$) are pairwise non-conjugate actions. 
By shrinking fibers (see \cite{CheegerGromov1990}, cf. \cite{CaiRong2009}) there is a continuous family of Riemannian metrics $g_k(\epsilon)$ on $S^2\times S^3$ collapsing to quotient manifold $S^2\times S^3/\rho_k(S^1)$ with sectional curvature bound $|\sec(S^2\times S^3,g_k(\epsilon))|\le C$, where $C$ is a constant independent of $k$. 

Thus, the nilpotent Killing structures $\mathfrak n_k$ of $g_k(\epsilon)$ are non-isomorphic, and by joining to $h_0$, collapsed metric $g_{1,\epsilon}$ can be changed smoothly to any $g_{k,\epsilon}$.

By Theorem \ref{thm-lip-stable}, the pairwise Lipschitz distance $d_L(g_i(\epsilon),g_j(\epsilon))$ goes to $\infty$ for $i\neq j$ as $\epsilon\to 0$. 
\end{example}

In general,  let $\mathcal{R}(M;\epsilon,d)$ be the moduli space endowed with Lipschitz distance, which consists of all isometric classes of $\epsilon$-collapsed Riemannian metrics on a $n$-manifold $M^n$, whose sectional curvature $|\operatorname{sec}|\le 1$ and diameter $\le d$. Let $\mathcal{R}(M;\epsilon,d,k)$ be the subspace whose underlying nilpotent Killing structure has dimension equals to $k$. Then by Theorem \ref{thm-lip-stable}, for $0<\epsilon\le\epsilon(n,d)$,
each component of $\mathcal{R}(M;\epsilon,d,k)$ corresponds a unique isomorphism class of nilpotent Killing structures on $M$. Moreover, the Lipschitz distance between metrics in $\mathcal{R}(M;\epsilon,d,k)$ corresponding to distinct nilpotent Killing structures goes to infinity as $\epsilon\to 0$.


The remaining of the paper is organized as follows.
In section 1 we fix some notations, and recall some preliminary facts on submersions and nilpotent Killing structures.
Since the proof of Theorem \ref{thm-lip-stable} is quite long, we first give an outline in section 2. 
Section 3 to section 6 are devoted to the proof of Theorem \ref{thm-lip-stable}.

\textbf{Acknowledgment.} We owe gratitude to Xiaochun Rong for raising related problems, his constant support and encouragement. The second author is grateful to Fuquan Fang for several highly stimulating conversations. We would like to thank Xuchao Yao for some useful discussions. This work is supported partially by NSFC Grant 11401398 and by Youth Innovative Research Team of Capital Normal University.



\section{Notations and Preliminaries}
In this section we fix some notations, and recall some elementary facts used later.

\subsection{Submersions}
Let $f:(M,g)\to (Y,h)$ be a (not necessarily Riemannian) submersion between two manifolds.
The $f$-vertical distribution tangent to $f$-fibers and its orthogonal complement, the $f$-horizontal distribution, are denoted by $\mathcal{V}_f$ and $\mathcal{H}_f$ respectively. We use   
$\mathcal V_f(x)$ (resp. $\mathcal{H}_f(x)$) to denote the vertical (resp. horizontal) subspace at $x\in M$.

The second fundamental form $\sndf_f$ of $f$-fibers and the integrability tensor $A_f$ of $f$ are defined respectively by
\begin{align*}
&\sndf_f:\mathcal V_f(x)\times \mathcal V_f(x)\to \mathcal H_f(x), \quad \sndf_f(T,T)=(\nabla_TT)^\perp|_x,\\
&A_f:\mathcal H_f(x)\times \mathcal H_f(x)\to \mathcal V_f(x), \quad A(X,Y)=[X,Y]^\top|_x\in \mathcal V_f(x).
\end{align*}

If a submersion $f:(M,g)\to (Y,h)$ is proper, then $(M,Y,f)$ forms a locally trivial fiber bundle, whose local trivialization can be realized via $f$-horizontal lifting curves. Since our construction of a bundle isomorphism in section 5 relies on this, we recall the local trivialization and related estimates in below.

Let $g$ and $h$ be any fixed Riemannian metric tensor on $M$ and $Y$ respectively. Let $p\in Y$ be a fixed point and let $0<r<\operatorname{inj.rad}(p)$ in $(Y,h)$. For any point $x\in f^{-1}(B_r(p))$, there is a unique minimal geodesic $\gamma_x:[0,1]\to Y$ connecting $f(x)=\gamma_x(0)$ and $p=\gamma_x(1)$. Because $f$ is proper, the horizontal lifting $\tilde \gamma_x:[0,1]\to M$ at $x$ is uniquely well-defined by $\tilde \gamma_x(0)=x$, $f(\tilde \gamma_x(t))=\gamma_x(t)$, and tangent vector $\tilde\gamma_x'(t)$ lies in horizontal distribution $\mathcal H_f$. 
We define a map 
\setcounter{equation}{0}
\addtocounter{theorem}{1}
\begin{equation}\label{def-varphi}
\varphi: f^{-1}(B_r(p))\to B_r(p)\times F_p \quad \text{by}\quad \varphi(x)=(f(x),\tilde\gamma_x(1)),
\end{equation} 
where $F_p=f^{-1}(p)$ is a $f$-fiber over $p$.

By construction, $\operatorname{pr}_1\circ \varphi=f$, where $\operatorname{pr}_1$ is the projection to the 1st factor.
Since $\varphi$ can be viewed as a projection of a flow in $TM$  generated by tangent fields of $\tilde \gamma_x$ at time $1$, 
$\varphi:f^{-1}(B_r(p))\to B_r(p)\times F_p$ is a diffeomorphism. Thus the map $\varphi$ is a local trivialization of fiber bundle $(M,Y,f)$. 

Let $\varphi_2:f^{-1}(B_r(p))\to F_p$ be the 2nd factor of $\varphi$, then by definition 
\begin{equation}\label{def-varphi2}
\varphi(x)=(f(x),\varphi_2(x)), \quad \varphi_2(x)=\tilde\gamma_x(1).
\end{equation}

By standard variation methods (cf. Lemma 1 in \cite{LiXu18}), $\dif \varphi_2$ is under control by $\sndf_f$, $A_f$, Lipschitz and co-Lipschitz constant of $f$, and the sectional curvature bound of the base space $Y$, as follows.

Let $L_0>0$ be a co-Lipschitz constant of $f$, i.e., 
for any horizontal vector $\xi\in \mathcal H_f$,  $|\xi|_{g}\le L_0\cdot |\dif f(\xi)|_h$.
Then for any vertical vector $v\in \mathcal V_f(x)$, 
\setcounter{equation}{0}
\addtocounter{theorem}{1}
\begin{equation}\label{lem-second-fundamental-co-Lip}
e^{-L_0 |\sndf|r(x)} |v| \le |\dif \varphi_2 (v)|\le e^{L_0 |\sndf|r(x)} |v|,
\end{equation}
where $|\sndf|=\sup_{q\in Y} |\sndf_{F_q}|$, and $r(x)=d(f(x),p)\le r$ is the distance between points $f(x)$ and $p$.

If $L_1$ is a Lipschitz constant of $f$ and $|\sec(Y,h)|\le 1$, then any horizontal $w\in \mathcal H_f(x)$ and $0\le r(x)\le \min\{\frac{\pi}{2},r\}$,
\begin{equation}\label{lem-integrability}
|\dif \varphi_2(w)|\le CL_0L_1^2e^{L_0|\sndf|r(x)}|A|r(x)\cdot |w|,
\end{equation}
where $|A|=\sup_{q\in Y} |A_{F_q}|$ and $C$ is a universal constant.

Now let $f:(M,g)\to (Y,h)$ be an $\epsilon$-almost Riemannian submersion, i.e., , for any vector $\xi$ perpendicular to a $f$-fiber,
\setcounter{equation}{0}
\addtocounter{theorem}{1}
\begin{equation}\label{eq-almost-riemsub}
e^{-\epsilon}|\xi|_{g}\le |\dif f(\xi)|_h\leq e^{\epsilon}|\xi|_g.
\end{equation}
By definition and easy calculation, the norms of $\sndf_f$ and $A_f$ are pointwisely bounded by the second fundamental form $\nabla^2f=\nabla df$.  That is,
\addtocounter{theorem}{1}
\begin{equation}\label{eq-C2-bound}
|\sndf_{F_p}|\leq e^\epsilon\cdot  |(\nabla df)_{F_p}|, \qquad
|A_{F_p}|\le 2e^{3\epsilon}\cdot |(\nabla df)_{F_p} |,
\end{equation}
where $|\nabla df|=\max_{|X|=|Y|=1} |\nabla df (X,Y)|$.

If $M$ is complete and $Y$ is connected, then for any $p,q\in Y$,
\begin{equation}\label{almost-equidistance}
e^{-\epsilon}\cdot d(p,q) \leq d_H(f^{-1}(p),f^{-1}(q))\leq e^{\epsilon}\cdot d(p,q),
\end{equation}
where $d_H(A,B)$ is the Hausdorff distance between two subsets $A,B$ in a metric space, i.e., the infimum of $\epsilon>0$ such that the $\epsilon$-neighborhood of $A$ contains $B$ and vice versa.

Indeed, because $f(M)$ is open and close in $Y$, $f(M)=Y$. By definition, $f$ is $e^\epsilon$-Lipschitz, which implies
$d(p,q)\le e^\epsilon\cdot  d(f^{-1}(p),f^{-1}(q))$. Because the horizontal lifting curve of a  minimal geodesic $\gamma$ connecting $p$ and $q$ has length $\le e^\epsilon d(p,q)$, $d_H(f^{-1}(p),f^{-1}(q))\le e^\epsilon\cdot d(p,q)$.

The injectivity radius function $\injrad_{g}:M\to \mathbb R$ on a complete Riemannian manifold $(M,g)$ is known to be locally $1$-Lipschitz \cite{Xu17}, i.e., for any two points $p,q$ in $M$,
\setcounter{equation}{0}
\addtocounter{theorem}{1}
\begin{equation}\label{inj-lip}
\injrad_g(q)\ge\min\{\injrad_g(p),\conjrad_g(q)\}-d(p,q),
\end{equation}
where $\conjrad_g(q)$ is the conjugate radius at $q$, and $d(p,q)$ is the distance between $p$ and $q$.

Given real numbers $L\ge 1$ and $\epsilon\ge 0$, a (not necessarily continuous) map $\psi:X\to Y$ between metric spaces is called a
\emph{$(L, \epsilon)$-quasi-isometry} if for all $x_1$ and $x_2$ in $X$,
\setcounter{equation}{0}
\addtocounter{theorem}{1}
\begin{equation}\label{def-quasi-isometry-1}
L^{-1}d_X(x_1,x_2)-\epsilon\le d_Y(\psi(x_1),\psi(x_2))\le Ld_X(x_1,x_2)+\epsilon,
\end{equation}
and $\psi(X)$ is $\epsilon$-dense in $Y$.

We use $\varkappa(\epsilon|a,b,c,\dots)$ to denote a positive  function depending on $\epsilon,a,b,c,\dots$ such that after fixing $a,b,c,\dots$, $\varkappa(\epsilon|a,b,c,\dots)\to 0$ as $\epsilon\to 0$. It will be simply written as $\varkappa(\epsilon)$, if the dependence is clear.

\subsection{Nilpotent Killing structures}
The reference for this subsection is \cite{CFG1992}.

Let $\mathfrak n$ be a sheaf of Lie algebras generated by locally defined smooth vector fields on $M$.
A metric $g$ on $M$ is called \emph{$\mathfrak n$-invariant}, if all (local) sections of $\mathfrak n$ are Killing fields for $g$.

For a local section $X$ of $\mathfrak n$, its flow defines a local one-parameter action. A set $Z\subset M$ is called \emph{invariant} if $Z$ is preserved by all such actions. For any point $p\in M$, the \emph{orbit} $O_p$ of $p$ is defined to be the minimal invariant set containing $p$.

Let $\mathfrak n$ be a sheaf of nilpotent Lie algebras. It is called a \emph{nilpotent Killing structure} for $g$, if for any $p\in M$, there is an invariant neighborhood $U$ of $p$ and a normal covering $\pi:\tilde U\to U$ such that
\addtocounter{theorem}{1}
\begin{enumerate}
	\item The integral of the pullback sheaf $\pi^*\mathfrak n(\tilde U)$ generates an isometric action $\rho$ of a simply connected nilpotent Lie group $N_U$, whose kernel $K=\ker\rho$ is discrete.
	\item $N_U$ and the deck-transformation group $\Lambda$ on $\tilde U$ generates an isometric action of a Lie group $H$ of finite many components, extending that of $\Lambda$ such that the identity component $N_0=N_U/K$.
	\item For any open $\tilde W\subset \tilde U$ containing a preimage point of $p$, the structure homomorphism $\pi^*\mathfrak n(\tilde U)\to \pi^*\mathfrak n(\tilde W)$ is an isomorphism.
	\item The neighborhood $U$ and covering $\tilde U$ can be chosen independent of $p\in O_p$.
\end{enumerate}
The $\mathfrak n$-invariant metric $g$ is called $(\rho,k)$-round, if in addition, the neighborhood $U$ above can be chosen to satisfies the following properties.
\addtocounter{theorem}{1}
\begin{enumerate}
	\item $U$ contains a metric ball $B_\rho (p)$ and all points in $\tilde U$ away from boundary have injectivity radius $>\rho$.
	\item $\#H/N_0=\#\Lambda/(\Lambda\cap N_0)\le k$.
\end{enumerate}

To illustrate what happens, we give some elementary but typical examples.

Let $N$ be a simply connected nilpotent Lie group, $\Lambda$ be a co-compact discrete subgroup, and let $g$ be a left invariant metric on $N$. Let $\mathfrak n$ be the sheaf of right invariant vector fields, which are Killing fields for $g$.
\begin{example}[nilmanifolds]
	Let $X$ be a right invariant vector field on $N$. Then for any $a\in N$ and $\lambda\in \Lambda$,
	$\lambda \cdot \exp tX\cdot a = \exp t\operatorname{Ad}_\lambda X \cdot \lambda \cdot a$. The conjugate quotient $\mathfrak n$ by $\Lambda$ defines a canonical nilpotent Killing structure on the nilmanifold $\Lambda\backslash N$. The center of $\mathfrak n$ will descend to a subsheaf which generates a torus action on $\Lambda\backslash N$.
\end{example}

Let $\nabla^{\operatorname{can}}$ be the canonical flat connection  with parallel torsion on $N$. Then its affine transformation group is $N\rtimes \operatorname{Aut}(N)$.
A \emph{infra-nilmanifold} $Z$ is a compact quotient manifold $\Gamma\backslash N$, where $\Gamma$ is a discrete subgroup of $N\rtimes \operatorname{Aut}(N)$.  The connection $\nabla^{\operatorname{can}}$ descends to $Z$, which is called a canonical \emph{affine structure}.  The affine group is denoted by $\operatorname{Af\!f}(\Gamma\backslash N)$. Since the index $[\Gamma:\Gamma/\Gamma\cap N]<w(n)$, the canonical nilpotent Killing structure on $(\Gamma \cap N)\backslash N$ induces a canonical nilpotent Killing structure on $Z$.

\begin{example}[affine bundles]
	A fiber bundle $(X,Y,f)$ is called to be \emph{affine}, if
	its fiber $Z$ is diffeomorphic to a infra-nilmanifold  $\Lambda\backslash N$ and its structure group is contained in $\operatorname{Af\!f}(\Gamma\backslash N)$. The sheaf of parallel vector fields along fibers naturally form a nilpotent Killing structure $\mathfrak n$ on $X$ (cf. \cite[II.4]{CFG1992}). A metric $g$ on $X$ is called \emph{affine-invariant}, if it is $\mathfrak n$-invariant. 
	
	Let $G$ be a compact group acting on $X$ and $Y$ isometrically. If the bundle projection $f$ is $G$-equivariant, i.e.,
	$$f(g\cdot x)=g\cdot f(x), \quad\text{for}\; \forall\, g \in G, \; \forall\, x\in X,$$
	and at the same time, $G$ preserves the affine structure on every fiber, then the action of $G$ extends to an action on $\mathfrak n$, such that the actions of $\mathfrak n$ and $G$ on $X$ commute (see \cite{CFG1992} for details). If the action of $G$ is free, then the quotient sheaf $\bar{\mathfrak n}$ on $X/G$ is also a nilpotent Killing structure.
	
	A bundle map $\Phi$ between two affine bundles $(X_i,Y_i,f_i)$ $(i=1,2)$ is called affine-equivariant, if it preserves the affine structures, i.e., $\Phi_*\mathfrak n_1$ is a subsheaf of $\mathfrak n_2$.
\end{example}

The existence of a nilpotent Killing structure was proved in \cite{CFG1992} for collapsed manifolds with bounded sectional curvature.
\begin{theorem}[\cite{CFG1992}]\label{thm-nilpotent-structure}
	For any $\delta>0$ and integer $n>0$, there are $\rho,\epsilon>0$, and integer $k\ge 1$ such that the following holds.
	
	Let $(M,g)$ be an $\epsilon$-collapsed $n$-manifold of $|\sec(M,g)|\le 1$. Then $M$ carries a nilpotent Killing structure $\mathfrak n$ of positive rank for a nearby $\mathfrak n$-invariant metric $g_\delta$, which is $(\rho,k)$-round and satisfies
	\begin{enumerate}
		\item $g^{-\delta}g<g_\epsilon<e^{\delta}g$,
		\item $g_\delta$'s connection is $\varkappa(\delta)$-close to that of $g$,
		\item curvature operator $R_{g_\delta}$ and its $i$-th covariant derivatives is bounded by $c(n,i,\delta)$.
	\end{enumerate}    
\end{theorem}

By its construction in \cite{CFG1992}, $\mathfrak n$ is induced by an $O(n)$-invariant nilpotent Killing structure on an $O(n)$-equivariant affine bundle $(FM,Y,f)$, where $FM$ is the  orthonormal frame of $M$ with a canonical metric induced by a bi-invariant metric on $O(n)$ and the Levi-Civita connection of $g$. 

Conversely, let $\mathfrak n$ be a nilpotent Killing structure on $(M,g)$ and $g$ is $\mathfrak n$-invariant. Then by definition, the differential of local actions of $\mathfrak n$ gives rise to a nilpotent Killing structure $\tilde {\mathfrak n}$ on $FM$, which corresponds to a canonical nilpotent structure on an $O(n)$-equivariant affine bundle, whose quotient is $\mathfrak n$.

We call two fiber bundles $(X_i,Y_i,f_i)$ $(i=1,2)$ to be \emph{isomorphic} if there are diffeomorphisms $\Phi:X_1\to X_2$ and $\Psi:Y_1\to Y_2$ such that $\Psi\circ f_1=f_2\circ \Phi$. Two affine bundles are \emph{isomorphic} if such $\Phi$ also preserves the affine structure. Clearly, $\Phi$ is affine if and only if $\Phi$ is $\mathfrak n$-equivariant, i.e., $\Phi$ commutes with the local actions of the canonical nilpotent Killing structure.

\section{Outline of proof Theorem \ref{thm-lip-stable}}
Our main Theorem \ref{thm-lip-stable} is a consequence of the following. 

\begin{theorem}\label{main-techthm}
	Given any real number $L_0\ge 1$, $\delta_0>0$, and positive integers $n>m\ge 1$, there is $\epsilon_0(L_0,\delta_0,n)>0$ (independent of $\lambda\ge 1$) such that the following holds.
	
	Let $M^n$ be a complete $n$-manifold and $(M^n,Y_i^m,f_i)$ ($i=1,2$) be affine bundles over $m$-manfolds $Y_i^m$ respectively. Let $g_i$ be an affine-invariant (w.r.t. $f_i$) Riemannian metric on $M$ such that 
	\begin{equation}\label{bounded-curv}
	|\sec(M,g)|\le 1, \quad |\sec(Y_i,h_i)|\le 1,
	\end{equation}
	where $h_i$ is the quotient metric of $g_i$ on $Y_i$.
	Assume that
	\begin{enumerate}
		\numberwithin{enumi}{theorem} 
		\item[$(\ref{main-techthm}.2)$]\label{Lip-equiv} the metrics $g_1,g_2$ are $L_0$-equivalent, i.e., $L_0^{-1}g_2\le g_1\le L_0g_2$, and
		\item[$(\ref{main-techthm}.3)$]\label{eps-collpase}
		for any $p\in Y_i$, the diameter and the second fundamental form of $f_i$-fiber $F_{i,p}=f_i^{-1}(p)$ satisfies $$\diam_{g_i} F_{i,p} \le \epsilon\cdot \min\{1,  \injrad_{h_i}(p)\}, \quad |\sndf_{F_{i,p}}|\le \delta_0.$$
	\end{enumerate}	
	Then there is an affine bundle isomorphism $(\Phi, \Psi)$ such that $f_2\circ\Phi = \Psi\circ f_1$.
	
	If in addition, there is a Lie group $G$ acting isometrically on both of $(M,g_i)$ and $(Y_i,h_i)$ so that $f_i$ is a $G$-equivariant affine bundle, then the diffeomorphisms in the bundle isomorphism between $(M,Y_i,f_i)$ are also $G$-equivariant.
\end{theorem}

Note that, in Theorem \ref{main-techthm} we do not assume $f_i$-fibers absorb all collapsed directions. Hence, potentially $Y_i$ maybe also collapse. 

Compared to the earlier stability results (\cite[section 7]{CFG1992}, \cite{Kapovitch2007Perelman}, \cite{Rong2012Stability}, \cite{LiXu18}, etc.), our main improvement here is that no fiberwise closeness (nor $C^1$-closeness) of $f_i$ are required.

Theorem \ref{thm-lip-stable} is a corollary of Theorem \ref{main-techthm}. Indeed, by (\ref{cond-Lip-equiv}), up to a shift of the collapsing scale, the nilpotent Killing structure $\mathfrak n_i$ $(i=1,2)$ corresponding to $g_i$ can be chosen of the same dimension. According to \cite{CFG1992} (or see section 2.2), there is a complete Riemannian manifold $Y_i$ and an $O(n)$-invariant affine bundle, $(FM,Y_i,\tilde f_i)$, on the orthogonal frame bundle $FM$, whose canonical nilpotent Killing structure descends $O(n)$-equivariantly to $\mathfrak n_i$ on $(M,g_i)$. Since there are nearby $\mathfrak n_i$-invariant metrics $g_{i,\epsilon}$ of uniformly bounded sectional curvature, without loss of generality we assume that $g_i$ (resp. the induced metric $\tilde g_i$ on $FM$) itself is $\mathfrak n_i$-invariant (resp. $O(n)$-invariant and $\mathfrak n_i$-invariant). Thus Theorem \ref{thm-lip-stable} is reduced to the stability of affine bundles with invariant metrics. Let $\tilde \Phi$ be the affine bundle isomorphism between $(FM, Y_i,\tilde f_i)$ provided by Theorem \ref{main-techthm}. Then its $O(n)$-quotient is the desired diffeomorphism in Theorem \ref{thm-lip-stable}.

The main part of this paper is devoted to the proof of Theorem \ref{main-techthm}, which is divided into three steps: 

Step 1. construct a $e^\epsilon L_0$-bi-Lipschitz diffeomorphism $\Psi:Y_1\to Y_2$ such that $d(\Psi\circ f_1,f_2)\le 2L_0\epsilon$. See Proposition \ref{prop-bi-Lip-transformation}.

Step 2. construct a diffeomorphism $\Phi_1:M\to M$ such that $\Phi_1$ is $2L_0^2\epsilon$-close to $\operatorname{Id}_M$ (measured in $g_1$), $f_2=\Psi\circ f_1\circ \Phi_1$, and for any $p\in Y_1$ and $F_{1,p}=f_1^{-1}(p)$, the restriction $\Phi_1|_{F_{2,p}}:F_{2,p}\to F_{1,\Psi(p)}$ is $e^\epsilon L_0$-bi-Lipschitz. See Proposition \ref{prop-bundle-diffeo}.

Step 3. modify $\Phi_1$ to get a bundle isomorphism $\Phi_2:M\to M$ that preserves the affine bundle structure. See Proposition \ref{prop-affine-isom}.

If in addition, $f_i$ is $G$-equivariant, where $G$ acts by isometries, then so are $\Phi_1$, $\Phi_2$ and $\Psi$.

The key in first two steps is a \emph{weak $C^1$-closeness} between $f_i$ in the following sense, whose proof will be carried out in Section 3.

\begin{proposition}[Weak $C^1$-closeness of $f_i$] \label{prop-distribution-close}  Let $f_i:(M^n,g_i)\to (Y_i^m,h_i)$ ($i=1,2$) be two $\epsilon$-Riemannian submersions satisfying (\ref{bounded-curv}), (\ref{Lip-equiv}.2), (\ref{eps-collpase}.3) and $|\nabla^2 f_i|\le \delta_0$ $(i=1,2)$. Then the followings hold.
	\begin{enumerate}
		\item\label{prop-vertical-close} For any $x\in M$, the dihedral angle measured in $g_i$ between vertical subspaces $\mathcal V_{f_1}(x)$ and $\mathcal V_{f_2}(x)$ $\le \varkappa(\epsilon\,|\,L_0, \delta_0, n)$.
		\item\label{prop-diff-close1} For any $f_1$-horizontal vector $w\in T_xM$ with $|w|_{g_1}=1$, 
		$$e^{-\varkappa(\epsilon|L_0,\delta_0, n)}L_0^{-1}\le |df_2(w)|\le  e^{\varkappa(\epsilon|L_0,\delta_0, n)}L_0.$$
	\end{enumerate}
\end{proposition}

For any $p\in Y_1$, after fixing a point $x\in f_1^{-1}(p)$, a smooth map $\psi_{p,x}$  can be defined by
\addtocounter{theorem}{1}
\begin{equation}\label{local-diffeo}
\psi_{p,x}:B_r(p)\to Y_2, \quad q \mapsto \psi_{p,x}(q)=f_2(\tilde \gamma_q(1)),
\end{equation}
where $0<r<\injrad_{h_1}(p)$, 
and $\tilde\gamma_q(t):[0,1]\to M$ is the $f_1$-horizontal lifting at $x$ of the unique minimal geodesic $\gamma_q:[0,1]\to Y_1$ from $p$ to $q$. 
By Proposition \ref{prop-distribution-close}, $\Psi_{p,x}$ is a
$L_0e^{\varkappa(\epsilon|L_0,\delta_0,n)}$-bi-Lipschitz diffeomorphism\footnote{This property roots back to an observation by Xiang Li and Xiaochun Rong, see \cite{Li2011phdthesis}.} from $B_r(p)$ onto an open set $V\subset Y_2$; see Remark \ref{rem-local-diffeo} below. 
The map $\Psi$ is essentially an average of local diffeomorphisms $\{\psi_{p,x}\}_{x\in f_1^{-1}(p)}$; see Section 4.

Our construction of $\Phi_1$ in Section 5 is via $f_1$-horizontal liftings of minimal geodesics, which is different from \cite[Proposition A.2.2]{CFG1992} by the normal projection on fibers through minimal geodesics. It requires only weak $C^1$-closeness and $C^0$-closeness of $\Psi\circ f_1$ and $f_2$. The method in \cite{CFG1992} still works with some additional arguments to guarantee $\Psi\circ \tilde f_1$ and $f_2$ to be $C^1$-close; see Remark \ref{rem-C1-close}.

In Section 6, we will further prove that, after identifying the simply connected nilpotent groups associated to $\mathfrak n_i$ by their lattice, their actions on the universal cover of local neighborhoods of points in $M$ are $C^1$-close; see Lemma \ref{lem-action-close}. Then $\Phi_2$ is obtained by the same average method in \cite{CFG1992}, i.e., averaging over a infra-nil fiber so that actions of the corresponding nilpotent group are conjugate on a local cover.

\section{Weak $C^1$-Closeness}\label{sec-weak-c1close}
In this section we prove Proposition \ref{prop-distribution-close}. 
Let $f_i:(M^n,g_i)\to (Y^m_i,h_i)$ $(i=1,2)$ be two $\epsilon$-almost Riemannian submersions which satisfy $|\sec(M,g_i)|\le 1$, $|\sec(Y_i,h_i)|\le 1$ and the following three conditions:
\addtocounter{theorem}{1}
\begin{enumerate}
	\numberwithin{enumi}{theorem}
	\item\label{def-collapse-lip-equiv} $g_i$ are $L_0$-equivalent, i.e., $L_0^{-1}g_2\le g_1\le L_0g_2$.
	\item\label{def-collapse-fiber} for any $p\in Y_i$, the intrinsic diameter of the $f_i$-fiber $F_{i,p}=f_i^{-1}(p)$ satisfies $\diam_{g_i} F_{i,p} \le \epsilon\cdot \min\{1,  \injrad_{h_i}(p)\}$.
	\item\label{def-collapse-2ndcontrol} the second fundamental form of $f_i$ satisfies $|\nabla^2 f_i|_{C^0} \le \delta$.
\end{enumerate}	

Note that no uniformly injectivity radius on $(Y_i,h_i)$ is assumed, and a prior there is no bound between $\injrad_{h_1}(f_1(x))$ and $\injrad_{h_2}(f_2(x))$ for a point $x\in M$. 

A key observation is that, after lifting $f_i$ to the iterated tangent space $T_o(T_xM)$ and blowing up the pull-back metrics, they would be close to two linear maps respectively, such that up to a diffeomorphic chart transformation, their fibers coincide with each other.

The proof is based on a quasi-isometry $\psi:(Y_1,h_1)\to (Y_2,h_2)$ such that $\psi\circ f_1$ is close to $f_2$, which naturally defined by a shift between $f_i:(M,g_i)\to (Y_i,h_i)$ $(i=1,2)$ below.

For $p\in Y_1$, let us define $\psi(p)$ to be a point in $f_2(f_1^{-1}(p))$.
Then by (\ref{def-collapse-lip-equiv}), for any $x\in M$ and $p=f_1(x)$, 
\addtocounter{theorem}{1}
\begin{equation}\label{close-by-quasi-isometry}
d(\psi(f_1(x)),f_2(x))\le  e^{\epsilon}\cdot \diam_{g_2}F_{1,p}\le e^\epsilon L_0\diam_{g_1}F_{1,p}.
\end{equation}
Moreover, it is easy to see that
\begin{align}\label{quasi-isometry-dist-1}
d(\psi(p),\psi(q))  \le  e^\epsilon L_0d_{H,g_1}(F_{1,p},F_{1,q})+e^\epsilon L_0 \diam_{g_1} F_{1,p},\quad \text{and}\qquad\\
d(\psi(p),\psi(q)) \ge 
e^{-\epsilon}L_0^{-1}d_{H,g_1}(F_{1,p},F_{1,q})-e^\epsilon \diam_{g_2}F_{2,p}-e^\epsilon L_0 \diam_{g_1}F_{1,p}.\notag
\end{align}
Since by (\ref{almost-equidistance}), $d_{H,g_1}(F_{1,p},F_{1,q})$ is proportional to $d(p,q)$ by $e^{\epsilon}$, we have
\begin{lemma}\label{lem-quasi-isom}
	The map $\psi$ is an $(e^{2\epsilon} L_0, 2e^{\epsilon}L_0 \epsilon)$-quasi-isometry.
\end{lemma}

Before given the proof of Proposition \ref{prop-distribution-close}, we make some preparation. 

Let us fix a point $x\in M$ and consider the exponential map of $(M,g_{2})$, $\exp_{x;g_{2}}:T_{x}M\to M$. Let
$g_{1}^*=\exp_{x;g_{2}}^*(g_{1})$ and $g^*_{2}=\exp_{x;g_{2}}^*(g_{2})$ be the pullback metric tensors on $T_{x}M$, and pull back again $g_{i}^*$ on $T_xM$ to the tangent space $T_{o}(T_xM)$ by  $\exp_{o;g_{1}^*}$ of $g_1^*$, where the pullback tensors are denoted by $g_{i}^{**}$.

Then the ball $B_{\frac{\pi}{2}}(o;g_1^*)\subset (T_o(T_xM),g_{1}^*|_{o})$, denoted again by $U$, satisfies (cf. \cite{Xu17}) 
\begin{equation}\label{C1-close-injrad}
\injrad(U,g^{**}_{i})\ge \frac{\pi}{2}L_0^{-1}, \quad   |\sec(U,g^{**}_{i})|\le 1.
\end{equation}
The lifting $\epsilon$-almost Riemannian submersions are well-defined on $U$,
\begin{equation}\label{def-lifting}
\tilde f_{i}=f_{i}\circ\exp_{x;g_2}\circ \exp_{o;g_1^{*}}: (U,g^{**}_{i})\to \tilde f_i(U)\subset (Y_i,h_i).
\end{equation}

Let $\{f_{i,j}:(M_j,g_{i,j})\to (Y_{i,j},h_{i,j})\}_{i=1,2}$ be a contradiction sequence to Proposition \ref{prop-distribution-close} with $\epsilon_j\to 0$, where the conclusion fails at a point $x_j\in M_j$.
Without loss of generality, we assume that $\injrad_{h_1}(f_{1,j}(x))\le \injrad_{h_2}(f_{2,j}(x)).$ Let
$\hat\varepsilon_j=\epsilon_j^{1/2}\cdot \injrad_{h_1}(f_1(x))$.

After blowing up with $\hat\varepsilon_j^{-1}$, by Cheeger-Gromov convergence theorem (\cite{Cheegerphdthesis,Cheeger1970Finiteness,GLP1981}, cf. \cite{GW1988,Peters1987,Kasue1989}) and (\ref{C1-close-injrad}), 
\begin{align*}
\left(U_j, o_j, \hat\varepsilon_j^{-2}g^{**}_{i,j}\right) \overset{C^{1,\alpha}}{\longrightarrow} \left(\mathbb{R}^n,g_{i,\infty}, o_i\right), \qquad j\to \infty,\\
\left(Y_{i,j},f_{i,j}(x_j),\hat\varepsilon_j^{-2}h_{i,j}\right) \overset{C^{1,\alpha}}{\longrightarrow} \left(\mathbb{R}^m,o_i\right), \qquad j\to\infty. 
\end{align*}
Then the identity map $I_j:\left(U_j, o_j, \hat\varepsilon_j^{-2}g^{**}_{1,j}\right)\to \left(U_j, o_j, \hat\varepsilon_j^{-2}g^{**}_{2,j}\right)$ converges to a smooth bi-Lipschitz map $I_\infty:(\mathbb R^n, g_{1,\infty}, o_1)\to(\mathbb R^n, g_{2,\infty}, o_2)$, which in general is not linear.

By passing to a subsequence, the lifting map defined by (\ref{def-lifting}), $$\tilde f_{i,j}:\left(U_j,o_j,\hat\varepsilon_j^{-2}g^{**}_{i,j}\right)\to 
\left(\tilde f_{i,j}(U_j),f_{i,j}(x_j),\hat\varepsilon_j^{-2}h_{i,j}\right),$$ converges to a canonical projection $\tilde f_{i,\infty}:\left(\mathbb R^n,g_{i,\infty},o\right)\to \left(\mathbb R^m, o_i\right).$

Let us consider the quasi-isometry $\psi_j:Y_{1,j}\to Y_{2,j}$ defined for $\{f_{i,j}\}$ such that $\psi_j$ maps $f_{1,j}(x_j)$ to $f_{2,j}(x_j)$. By (\ref{close-by-quasi-isometry}) and after blow-up, the distance error measured on $(Y_{2,j}, \hat\varepsilon_j^{-2}h_{2,j})$ satisfies that, for any $\tilde y\in U_j$,
\begin{equation*}
	d(\psi_j\circ \tilde f_{1,j}(\tilde y), \tilde f_{2,j}(\tilde y))\le e^{\epsilon_j} L_0\cdot\hat\varepsilon_j^{-1}  \diam_{g_{1,j}}f_{1,j}^{-1}(\tilde f_{1,j}(\tilde y)).
\end{equation*}
Moreover, for any $p,q\in (Y_{1,j},\hat\varepsilon_j^{-2}h_{1,j})$, we derive from  (\ref{quasi-isometry-dist-1}) that
\begin{equation*}
\begin{aligned}
&d\left(\psi_j(p),\psi_j(q)\right)\\
\le\; & e^{2\epsilon_j}L_0\cdot d\left(p,q\right)+
e^{2\epsilon_j}L_0 \cdot 
\hat \varepsilon_j^{-1} \cdot \min \left \{\diam_{g_{1,j}}f_{1,j}^{-1}(p), \diam_{g_{1,j}}f_{1,j}^{-1}(q)\right\}.
\end{aligned}
\end{equation*}

By the choice of $\hat\varepsilon_j$ and $|\nabla^2 f_{i,j}|\le \delta_0$, $$\hat\varepsilon_j^{-1}\cdot \diam_{g_{1,j}}f_{1,j}^{-1}(\tilde f_{1,j}(\tilde y))\le \sqrt{\epsilon_j}.$$ 

It follows from a standard diagonal procedure that a subsequence of $\psi_j$ converges to a $L_0$-Lipschitz map $\psi_{\infty}:(\mathbb R^n,o_1)\to (\mathbb R^n,o_2)$ such that 
\begin{equation}\label{limit-coincide}
	\phi_\infty\circ \tilde f_{1,\infty}=\tilde f_{2,\infty}\circ I_\infty.
\end{equation}
		
We are now ready to prove Proposition \ref{prop-distribution-close}.

\begin{proof}[Proof of Proposition \ref{prop-distribution-close}]
	 Recall that the dihedral angle between the vertical subspaces $\mathcal V_{f_1}(x)$ and $\mathcal V_{f_2}(x)$ is defined to be the Hausdorff distance $d_H(SV_{f_1}(x),SV_{f_2}(x))$ in the $g_2$-unit sphere of $T_xM$, where $SV_{f_i}(x)=\mathcal V_{f_i}\cap S_xM$ and $S_xM$ is the $g_2$-unit sphere centered at the origin of $T_xM$. Because $g_1$ and $g_2$ are $L_0$-equivalent, it makes no substantial difference, if instead, the dihedral angle is measured in $g_1$.
	
	Since $|\nabla^2f_i|\le \delta_0$, it is clear that the contradicting sequence above is $\varkappa(\epsilon)$-close to $\tilde f_{i,\infty}$ in the $C^{1,\alpha}$-norm, whose fibers, by (\ref{limit-coincide}), coincide with each other. So by a contradiction argument, we derive (\ref{prop-vertical-close}).
	
	Since the pullback metric $I_\infty^*g_{2,\infty}$ satisfies $$L_0^{-1}I_\infty^*g_{2,\infty}\le g_{1,\infty}\le L_0I_\infty^*g_{2,\infty},$$ by (\ref{limit-coincide}),
	for any $\tilde f_{1,\infty}$-horizontal vector $w\in T_{o}\mathbb R^m$ with $|w|_{g_{1,\infty}}=1$, 
	$$L_0^{-1}\le |\dif \tilde f_{2,\infty}(w)|\le L_0.$$
	By $C^{1,\alpha}$-convergence of $\tilde f_{i,j}$ again, (\ref{prop-diff-close1}) follows from a contradiction argument.
\end{proof}

\begin{remark}\label{rem-non-C1-close}
	Clearly, (\ref{prop-diff-close1}) implies a uniform control on the deviation of $f_i$-horizontal distributions from each other.
	However, they are not necessarily close. An easy example can be found on a flat torus, where the two metrics $g_1$ and $g_2$ are induced from $(\mathbb R^2,\tilde g_i)$ with two flat metrics whose orthonormal decompositions are different from each other by a definite angle. 
\end{remark}

\section{Diffeomorphism $\Psi$ between Base Spaces}\label{sec-transformation}

We are to improve the quasi-isometry $\psi$ in Lemma \ref{lem-quasi-isom} to a bi-Lipschitz diffeomorphism $\Psi:Y_1\to Y_2$ via center of mass. 

Let $(F, \nu)$ be a probability measure space and let $\imath:F\to Y_2$ be a measurable map into $(Y_2,h_2)$.  If its image $\imath(F)$ is contained in a convex ball $B_{a}(z)$ of radius $a< \frac{\pi}{6}$, then the smooth energy function $$E(y)=\frac{1}{2}\int_{F}d^2(\imath(x),y)d\nu$$ is strictly convex in $B_{3a}(z)$. It is clear that $E$ takes a unique minimum point at some point $z_1$ in the closure of $B_{2a}(z)$. We call $z_1$  the \emph{center of mass of $\imath$} (cf. \cite{GroveKarcher1973}, \cite{CFG1992}). 

\begin{proposition}\label{prop-bi-Lip-transformation}
	There is $\epsilon_0=\epsilon_0(L_0,\delta_0,n)>0$ such that for any two $\epsilon$-almost Riemannian submersions $f_i:(M,g_i)\to (Y_i,h_i)$ in Proposition \ref{prop-distribution-close} with $\epsilon<\epsilon_0$, there is a $e^{\varkappa(\epsilon|L_0,\delta_0,n)}L_0$-bi-Lipschitz diffeomorphism $\Psi:(Y_1,h_1)\to (Y_2,h_2)$ satisfying that,
	for any $p\in Y_1$ and $x\in F_{1,p}=f_1^{-1}(p)$,
	\begin{equation}\label{ineq-dist-error1}
	d(\Psi\circ f_1(x),f_2(x))\le 2L_0\diam_{g_1}F_{1,p}.
	\end{equation}
	If in addition, $f_i$ $(i=1,2)$ are $G$-equivariant, then $\Psi$ is also $G$-equivariant.  
\end{proposition}
\begin{proof}
Since the diameter of $f_2(F_{1,p})$, $$\diam_{h_2}(f_2(F_{1,p}))\le L_0e^\epsilon\epsilon<\frac{\pi}{6},$$ let us define $\Psi(p)$ to be the center of mass of $f_2:F_{1,p}\to (Y_2,h_2)$. By definition, (\ref{close-by-quasi-isometry}) implies (\ref{ineq-dist-error1}). In the following we prove that $\Psi$ is a diffeomorphism.
		
	By definition, $\Psi(p)$ is the critical point of the energy functional $E_p(\cdot)=E(p;\cdot)$, where
	\begin{equation}
	\begin{aligned}
	E(p;y)&=\frac{1}{2}\intbar_{F_{1,p}}d^2(f_2(x),y)\dif \operatorname{vol}(x)\\
	&=\frac{1}{2\operatorname{vol}(F_{1,p})} \int_{F_{1,p}}d^2(f_2(x),y)\dif \operatorname{vol}(x)
	\end{aligned}
	\end{equation} 
	Let $(p^1,p^2,\dots,p^m)$ and $(y^1,y^2,\cdots,y^m)$ be local coordinates around $p$ and $\Psi(p)$ respectively. Then 
	$y=\Psi(p)$ is the implicit function determined by the Pfaffian equation, $\partial_y E(p;y)=0$, a system of equations in local coordinates on $Y_1\times Y_2$ below:
	\begin{equation}\label{eq-implicit-function}
	\begin{aligned}
	&E_j'(p^1,\dots,p^m;y^1,\dots,y^m)\\
	=&\intbar_{F_{1,p}} d(f_2(x),y)\left< \nabla r_{f_2(x)},\frac{\partial}{\partial y^j}\right>
	\dif \operatorname{vol}(x)=0, \quad (j=1,\dots,n)
	\end{aligned}
	\end{equation}
	where $r_y=d(y,\cdot)$ be the distance function on $(Y_2,h_2)$. 
	
	Let $E'=(E_1',\dots, E_m')$ and $\partial_y E'$ be the differential of $E'(p;\cdot)$ after fixing $p$. Then it is easy to determine $d \Psi$ by 
	\begin{equation}\label{trans-implicit-dif}
	d \Psi=-(\partial_y E')^{-1}\circ \partial_p E',
	\end{equation} 
	As for $\partial_pE'$, we may assume that $(y^1,\dots,y^m)$ is the normal coordinates at $y\in (Y_2,h_2)$. Then by identifying points and their position vectors, $$d(f_2(x),y)\nabla r_{f_2(x)}=-f_2(x).$$ Thus (\ref{eq-implicit-function}) can be rewritten as a vector equation
	$$E'(p^1,\dots,p^m;0,\dots,0)
	=\intbar_{F_{1,p}} -f_2(x)\dif\operatorname{vol}(x)=0.$$
	For any unit-speed geodesic $\gamma(t)$ with $\gamma(0)=p$, let $x(t)\in F_{1,\gamma(t)}$ be its $f_1$-horizontal lifting that starts at $x\in F_{1,p}$. By direct calculation,
	\begin{align}
	\frac{\partial E'}{\partial t}=&\frac{d}{dt}\intbar_{F_1}-f_2(x(t))\dif\operatorname{vol}(t)
	\nonumber\\
	=&-\intbar \dif f_2(x'(t))-\intbar f_2(x(t))H_{x'(t)}+\intbar f_2(x(t)) \intbar H_{x'(t)},\label{trans-cal}
	\end{align}
	where $H_{x'(t)}$ is the mean curvature of $F_{1,\gamma(t)}$ along horizontal vector $x'(t)$.
	
	By (\ref{prop-diff-close1}), the vector $\intbar_{F_{1,p}}\dif f_2(x'(0))$ has norm in $$\left[e^{-\varkappa_1(\epsilon|L_0,\delta_0,n)}L_0^{-1}, e^{\varkappa_1(\epsilon|L_0,\delta_0,n)}L_0\right].$$ At the same time, by (\ref{eq-almost-riemsub}) and (\ref{def-collapse-2ndcontrol}), 
	$$\left|f_2(x(t))H_{x'(t)}\right|\le (m-n)L_0e^{\epsilon}\delta\epsilon.$$ 
	So is the last term in (\ref{trans-cal}).
	
	Then for sufficient small $\epsilon$, $\left.\frac{\partial E'}{\partial t}\right|_{t=0}$ has norm 
	\begin{equation}\label{trans-implicit-dif-1}
	\left|\frac{\partial E'}{\partial t}\right|_{t=0}\in 
	\left[e^{-\varkappa_2(\epsilon)}L_0^{-1}, e^{\varkappa_2(\epsilon)}L_0\right].
	\end{equation}

	On the other hand, $\partial_y E'$ equals to the Hessian of $E_p(y)$, which by standard Hessian comparison, satisfies 
	\begin{equation}\label{trans-implicit-dif-2}
	\cos(L_0^{-1}e^{-\epsilon}\epsilon)h_2\le \operatorname{Hess}(E_p)\le \cosh (L_0e^{\epsilon}\epsilon)h_2.
	\end{equation}
	
	Combining (\ref{trans-implicit-dif}), (\ref{trans-implicit-dif-1}) and (\ref{trans-implicit-dif-2}), we conclude that $d \Psi$ has norm 
	$$|\dif \Psi|\in \left[\left(e^{-\varkappa(\epsilon)}L_0\right)^{-1}, e^{\varkappa(\epsilon)}L_0\right].$$
	
	Now it is easy to see that $\Psi$ is a $e^{\varkappa(\epsilon)}L_0$-bi-Lipschitz diffeomorphism and satisfies the requirements in Proposition \ref{prop-bi-Lip-transformation}. Indeed, by construction $\Psi$ is also an $(e^{2\epsilon}L_0, 2e^\epsilon L_0\epsilon)$-quasi-isometry. It follows from (\ref{def-quasi-isometry-1}) that the fiber of $\Psi$ has diameter $\le 2e^{3\epsilon}L_0\epsilon$.
	Since (\ref{def-varphi}) is a local trivialization, the fiber of $\Psi$ must be connected. Hence $\Psi$ is a bi-Lipschitz diffeomorphism.
	
	If the fiber bundles $f_i$ $(i=1,2)$ are $G$-equivariant, then by the construction above, it is clear that $\Psi$ is also $G$-equivariant.  	
\end{proof}

\begin{remark}\label{rem-C1-close}
	If in addition, the higher derivatives of $f_i$ and the curvature tensor of $Y_2$ admit uniform bounds, then so is for $\Psi:(Y_1,h_1)\to (Y_2,h_2)$ in Proposition \ref{prop-bi-Lip-transformation}. In particular, $f_1$ and $\Psi^{-1}\circ f_2$ would be $\varkappa(\epsilon)$-$C^1$-close.
\end{remark}

\section{Bundle Isomorphism $\Phi_1$ on Total Spaces}\label{sec-bundle-map}

In this section we construct a diffeomorphic bundle map $$\Phi_1:(M,Y_1,\Psi^{-1}\circ f_2)\to (M,Y_1,f_1).$$
Continue from section \ref{sec-transformation}. Let  $\Psi:(Y_1,h_1)\to (Y_2,h_2)$ be the bi-Lipschitz diffeomorphism  provided by Proposition \ref{prop-bi-Lip-transformation}. Then the composition $\hat f_2=\Psi^{-1}\circ f_2: (M,g_1)\to (Y_1,h_1)$ is a $(\varkappa(\epsilon)+2\ln L_0)$-almost Riemannian submersion, which by (\ref{ineq-dist-error1}) satisfies 
\addtocounter{theorem}{1}
\begin{equation}\label{dist-error}
d(f_1(x),\hat f_2(x))\le 2L_0^2 \diam_{h_1}F_{1,f_1(x)}.
\end{equation}
Throughout this section, we only use $g_1$ and $h_1$ and all norms are measured by them.

For any $p\in Y_1$ and $x\in F_{2,p}=\hat f_2^{-1}(p)$, let $p_x= f_1(x)$. Up to a blowup rescaling, we assume that $\injrad_{h_1}(p)=1$.
Then by (\ref{dist-error}), (\ref{def-collapse-fiber}) and (\ref{inj-lip}), 
$$d(p_x,p)\le 2L_0^2\epsilon \cdot \injrad_{h_1}(p_x)\le\frac{2L_0^2\epsilon}{1-2L_0^2\epsilon}\injrad_{h_1}(p).$$
Thus, the minimal geodesic $\gamma_x:[0,1]\to Y_1$ from $p_x=\gamma_x(0)$ to $p=\gamma_x(1)$ in $(Y_1,h_1)$ is unique and  depends smoothly on $x$. Let $\tilde \gamma_x:[0,1]\to (M,g_1)$ be the unique $f_1$-horizontal lifting of $\gamma_x$ at $x$, then we define $\Phi_1(x)=\tilde \gamma_x(1)$. 

Clearly, $\Phi_1$ depends smoothly on $x$ and lies in $F_{1,p}$, and thus it is a bundle map from $(M,Y_1,\hat f_2)$ to $(M,Y_1,f_1)$, i.e., $f_1\circ\Phi_1=\hat f_2$.

Notice that if $\dif\Phi_1$ is non-degenerated at every point $x\in M$, then $\Phi_1$ is a covering map homotopic to the identity, and hence a diffeomorphic bundle isomorphism.  Indeed,  a smooth homotopy $H$ is naturally defined by 
\begin{equation}\label{def-homotopy}
H:[0,1]\times {M}\to M, \qquad H(t,x)=\tilde \gamma_x(t),
\end{equation}
such that $H(0,\cdot )=\operatorname{Id}_{M}$, and $H(1,\cdot)=\Phi_1$.

To show $d\Phi_1$ is isomorphic, we have the following observation.
\begin{lemma}\label{lem-bundle-isomorphism}
	$d\Phi_1$ is non-degenerate if and only if $d\varphi_2(v)\neq 0$, for any $f_2$-vertical vector $v$.
\end{lemma}
\begin{proof}
	Let $0<r<\injrad_{h_1}(p)=1$, and let $$\varphi:f_1^{-1}(B_{r}(p))\to B_{r}(p)\times F_{1,p}, \quad \varphi=(f_1(x),\varphi_2(x))$$ be a local trivialization of $f_1$ centered at $F_{1,p}$ defined in (\ref{def-varphi}), where $\varphi_2$ is its 2nd factor as in (\ref{def-varphi2}).
	Then by definition, 
	\begin{equation}
	\Phi_1(x)=\tilde\gamma_x(1)=\varphi_2(x), \quad \text{for any $x\in F_{2,p}\cap f_1^{-1}(B_{r}(p))$}.
	\end{equation}
	It follows that  
	\begin{equation}\label{Phi-isom-1}
	\dif \Phi_1(v)=\dif \varphi_2(v), \quad\text{for any $f_2$-vertical vector $v\in \mathcal V_{f_2}(x)$}.
	\end{equation}
	Furthermore, by definition
	\begin{equation}
	\dif \hat f_2=\dif f_1\circ \dif\Phi_1,
	\end{equation}
	which implies that
	the kernel of $\dif\Phi_1$ is contained in the $f_2$-vertical distribution $\mathcal V_{f_2}$, 
	\begin{equation}\label{Phi-isom-2}
	\ker\dif \Phi_1\subset \mathcal V_{f_2}(x).
	\end{equation} By (\ref{Phi-isom-1}) and (\ref{Phi-isom-2}) we conclude Lemma \ref{lem-bundle-isomorphism}.
\end{proof}

\begin{proposition}\label{prop-bundle-diffeo}
	There is $\epsilon_0=\epsilon_0(L_0,\delta_0,n)>0$ such that if $0\le\epsilon\le \epsilon_0$, then
	$\Phi_1:(M,g_1)\to (M,g_1)$ is a diffeomorphism satisfying $f_1\circ\Phi_1=\hat f_2$, and the restriction of $\Phi_1$ on every fiber is $e^{\varkappa(\epsilon)}$-bi-Lipschitz.
\end{proposition}
\begin{proof}
	Continue from the above discussion. Let $\varphi=(f_1,\varphi_2)$ be the local trivialization of $f_1$ be defined as (\ref{def-varphi2}). By Lemma \ref{lem-bundle-isomorphism}, it suffices to show that for any $g_1$-unit vector $v\in \mathcal V_{f_2}(x)$, $\dif \Phi_1(v)=\dif \varphi_2(v)\neq 0$. 
	
	Let $v=v^\bot + v^\top$ be its orthogonal decomposition such that $v^\bot\in \mathcal H_{f_1}(x)$ and $v^\top\in \mathcal V_{f_1}(x)$.
	By (\ref{prop-vertical-close}) in Proposition \ref{prop-distribution-close}, 
	\begin{equation}\label{ineq-tech-weaker-1}
	|v^\bot|\le \varkappa(\epsilon),\quad |v^\top|\ge \sqrt{1-\varkappa^2(\epsilon)},
	\end{equation}
	where $\varkappa(\epsilon)=\varkappa(\epsilon|L_0,\delta_0,  m)$ and the norm hereafter is measured in $g_1$.
	
	Since $|\sec(Y_1,h_1)|\le 1$, $\dif \varphi_2$ can be explicitly estimated by variation of horizontal curves. Let $\gamma:[0,1]\to (Y_1,h_1)$ be the minimal geodesic from 
	$p_x=f_1(x)$ to $p=\hat f_2(x)$. By (\ref{lem-second-fundamental-co-Lip})-(\ref{lem-integrability}) ,
	\begin{equation}\label{ineq-tech-weaker-2}
	\begin{aligned}
	&|\dif \varphi_2(v^\top)|\ge e^{-e^\epsilon\cdot |\sndf_{F_{1,\gamma}}|\cdot d(p_x,p)} \cdot |v^\top|,\\
	&|\dif \varphi_2(v^\perp)|\le
	Ce^{3\epsilon+e^\epsilon | \sndf_{1,\gamma}|\cdot d(p_x,p)}\cdot |A_{1,\gamma}|\cdot d(p_x,p)\cdot |v^\perp|,
	\end{aligned}
	\end{equation}
	where $|\sndf_{1,\gamma}|=\max_t |\sndf_{f_1^{-1}(\gamma(t))}|$ and $|A_{1,\gamma}|=\max_t |A_{f_1^{-1}(\gamma(t))}|$, which by (\ref{eq-C2-bound}), is bounded by $2e^{3\epsilon}\delta_0$.
	
	Since $d(p_x,p)\le \frac{2L_0^2\epsilon}{1-2L_0^2\epsilon}$, combing (\ref{ineq-tech-weaker-1})-(\ref{ineq-tech-weaker-2}) we derive
	\begin{align}
	|\dif \Phi_1(v)|
	&\ge |\dif \varphi_2(v^\top)|-
	|\dif \varphi_2 (v^\bot)|\notag\\
	&\ge e^{-\varkappa(\epsilon)}\sqrt{1-\varkappa(\epsilon)}-\varkappa(\epsilon)\label{ineq-fiber-proportional}.
	\end{align}
	Clearly, $|\dif\Phi_1|$ admits a similar upper bound.
\end{proof}

We make several remarks on Proposition \ref{prop-bundle-diffeo} in order.
\begin{remark}\label{rem-diameter-proportional}
	For any point $x\in M$, by Proposition \ref{prop-bi-Lip-transformation}, $\injrad_{h_1}(f_1(x))$ is also uniformly proportional to $\injrad_{h_1}(f_2(x))$.
	By Proposition \ref{prop-bundle-diffeo},  $\diam_{g_1}F_{1,f_1(x)}$ and $\diam_{g_2}F_{2,f_2(x)}$ are $e^{\varkappa(\epsilon)}L_0$-proportional to each other. 
\end{remark}
	
\begin{remark}\label{rem-isotopy}
	The method in this section can be applied as a replacement of \cite[Proposition A.2.2]{CFG1992} in the construction of a global nilpotent Killing structure in \cite{CFG1992}, where an isotopy from $\operatorname{Id}_M$ to $\Phi_1$ is required.
	
	Indeed, in the above case $f_1$ and $\hat f_2$ are $C^1$-close (see Remark \ref{rem-C1-close}). Hence, for any $t\in[0,1]$, the map $H_t:M\to Y_1$, $H_t(x)=\gamma_x(t)$ is also an almost Riemannian submersion, where $\gamma_x:[0,1]\to Y_1$ is the unique minimal geodesic from $f_1(x)$ to $\hat f_2(x)$. By the fact again that everything involved is $C^1$-close, it is easy to see that the estimate (\ref{ineq-tech-weaker-1})-(\ref{ineq-tech-weaker-2}) also works for $\Phi_t(x)=\tilde \gamma_x(t)$. Hence the homotopy $H:[0,1]\times M\to M$, $H(t,x)=\Phi_t(x)$ is an isotopy from $\operatorname{Id}_M$ to $\Phi_1$.
	
	Therefore, Proposition \ref{prop-bundle-diffeo} generalizes Proposition A2.2 in \cite{CFG1992}. One benefit of our approach is, the normal injectivity radius of fibers are not required to admit a uniform lower bound.
	 
\end{remark}

\begin{remark}\label{rem-local-diffeo}
Behind the proof of Proposition \ref{prop-bi-Lip-transformation} is the fact that $\psi_{p,x}$ defined in (\ref{local-diffeo}) is a diffeomorphism onto its image (though it is not explicitly used). This fact can be verified as follows.
	
Let $\varphi=(f_1,\varphi_2)$ be the local trivialization of $f_1$ given by (\ref{def-varphi}).
By the same argument in proving Proposition \ref{prop-bundle-diffeo}, $\mathcal V_{f_{2}}$ is transversal to the kernel of $\dif \varphi_2$ in $f_1^{-1}(B_r(p))$ for any $0<r<\min\{1,\injrad_{h_1}(p)\}$ and sufficient small $\epsilon$. By Lemma \ref{lem-bundle-isomorphism}, the map $\psi_{p,x}$ is a local diffeomorphism.

We claim that any $f_2$-fiber in $f_1^{-1}(B_r(p))$ intersects with $S_x=\varphi_2^{-1}(x)$ at most once. Hence $\psi_{p,x}$ is a diffeomorphism onto its image.

To verify the claim, let us argue by contradiction. If some $f_2$-fiber $F_2$ lying in $f_1^{-1}(B_r(p))$ intersects with  $S_x$ at two points $z_1$ and $z_2$, then there are two curves $\alpha:[0,1]\to S_x$ and $\beta:[0,1]\to F_2$, both of which are connecting $z_1$ and $z_2$, such that they are homotopic to each other with fixed endpoints. After passing to the tangent space of $T_xM$, the lifts of $S_x$ can be viewed as a coordinate plane $\tilde S_x$. It follows that the lifting $\tilde \beta$ has two endpoints in $\tilde S_x$, and thus there is some $t_0\in (0,1)$ such that $\tilde \beta'(t_0)$ is tangent to $\tilde S_x$, which contradicts to the fact that $\mathcal V_{f_2}$ is transversal to $S_x$.
\end{remark}

\section{Affine Bundle Isomorphism and Proofs of Main Theorems}
We first prove Theorem \ref{main-techthm}.

Let $(M,Y_i,f_i)$ be two affine bundles equipped with metrics $g_i, h_i$ on $M$ and $Y_i$ respectively, which satisfy the conditions in Theorem \ref{main-techthm}. Then $f_i:(M,g_i)\to (Y_i,h_i)$ are Riemannian submersions and $h_i$ is the quotient of $g_i$ by the canonical nilpotent Killing structure of $f_i$ on $(M,g_i)$. 

Because $f_i$ is a Riemannian submersion, $|\nabla^2f_i|$ is bounded by $\max\{|\sndf_{f_i}|,|A_{f_i}|\}$, where $|\sndf_{f_i}|$ is bounded by (\ref{eps-collpase}.3), and by O'Neil's formula \cite{ONeill1966} and (\ref{bounded-curv}), $|A|^2\le \frac{8}{3}$. Proposition \ref{prop-distribution-close} holds for $f_i$.

By Proposition \ref{prop-bi-Lip-transformation} and Proposition \ref{prop-bundle-diffeo}, there are diffeomorphic bundle isomorphism $(\Phi_1,\Psi^{-1})$ between $(M,Y_2,f_2)$ and $(M,Y_1,f_1)$ i.e., $\Psi^{-1}\circ f_2=f_1\circ \Phi_1$. However, $\dif\Phi_1$ generally does not preserve the affine connection between $f_i$-fibers.

In order to improve $\Phi_1$ to an affine bundle isomorphism, we first prove that the group actions induced by the affine structures are $C^1$-close.

Let $\mathfrak n_1$ be the canonical nilpotent Killing structure for affine bundle $f_1$, and $\mathfrak n_2$ be the push forward of that for $f_2$ by $\Phi_1$ on $M$. Let us fix a point $p\in Y_1$, and let $r=\injrad_{h_1}(p)/2$. Then $U=f_{1}^{-1}(B_r(p;h_1))$ is $\mathfrak n_i$-invariant. Let $\tilde U\overset{\pi}{\to}U$ be the universal cover, $\tilde f_1=f_1\circ \pi$ and $\tilde f_2=f_2\circ \Phi_1^{-1}\circ \pi$. Then the pullback $\tilde n_i$ on $\tilde U$ generates two free actions $\rho_i$ of simply connected nilpotent Lie groups $N_i$, which are left translations on $\tilde f_i$-fibers. Let $\Lambda$ be the fundamental group of $f_1^{-1}(p)=(\Psi^{-1}\circ f_2\circ\Phi_1^{-1})^{-1}(p)$. By Malcev's rigidity theorem (see \cite{Rag79}), $N_1$ and $N_2$ can be identified to a same group $N$ by the natural isomorphism between their lattice $\Lambda\cap N_i$. Moreover, the two actions $\rho_1$ and $\rho_2$ of $N$ coincide on $\Lambda$.
\begin{lemma}\label{lem-action-close}
	The two actions $\rho_i$ $(i=1,2)$ generated by the pullback $\tilde {\mathfrak n_i}$ on $\tilde U$ are $\varkappa(\epsilon|L_0,\delta_0,n)$-$C^1$-close.
\end{lemma} 
\begin{proof}
	Let $\tilde g_1$ (resp. $\tilde g_2$) be the pullback metric of $g_1$ (resp.  $(\Phi_1^{-1})^*g_2$) on $\tilde U$. The action $\rho_i$ is isometric with respect to $\tilde g_i$. Since $\sndf_{\tilde f_i}$ with respect to $\tilde g_i$ is under control, by \cite[Proposition 4.6.3]{BuserKarcher81} (or \cite[Lemma 7.13]{CFG1992}), $$\injrad_{\pi^*h_i}(\tilde x)\ge \min\{r/2,i_0(\delta_0,n)\}>0,$$ for any point $\tilde x$ with $d(\tilde x,\partial \tilde U)>r/2$. 
	
	Let $\hat\varepsilon=\min\{1,\injrad_{h_1}(p)\}$.
	Let us rescale $\tilde g_i$ by $\hat\varepsilon^{-1}$ and let $\epsilon\to 0$. By passing to a subsequence, we can assume that
	the equivariant convergence $$(\tilde U, \hat\varepsilon^{-1}\tilde g_i, x, \rho_i(N))\overset{C^{1,\alpha}}{\longrightarrow} (\tilde U, \tilde g_{i,\infty}, x, \rho_{i,\infty}(N)).$$ 
	Because the diameter of both $f_1$-fibers and $f_2$-fibers goes to $0$, the action of $\Lambda$ becomes more and more dense such that $\rho_{i,\infty}(N)$ is also the limit action of $\Lambda$. Thus the two limit actions coincide. This implies that the actions of $\tilde{\mathfrak n_i}$ are $C^1$-close, which are invariant under rescaling by $\hat\varepsilon$.
\end{proof}
By the $C^1$-closeness of $\rho_i$, it follows from the argument in \cite[section 7]{CFG1992} that $\Phi_1$ can be modified to an affine bundle isomorphism. In the following we give a proof for completeness.
\begin{proposition}\label{prop-affine-isom}
	There is an $G$-equivariant diffeomorphism $\Phi_2:M\to M$ such that $(\Phi_2)_*\mathfrak n_2=\mathfrak n_1$.
\end{proposition}
\begin{proof}
	We will follow the argument in \cite[section 7]{CFG1992} to derive $\Phi_2$. Let $\tilde U_1=\tilde f_1^{-1}(B_{r/2}(p))\subset \tilde U$.	
	
	We first prove that there is a $N$-equivariant diffeomorphism $\tilde \Phi_2:\tilde U_1\to \tilde U_1$.
	Continue from Lemma \ref{lem-action-close}. Let $\rho_i$ be the action of $N$ on $\tilde U$. For any $\lambda \in \Lambda$, we have $\rho_1(\lambda)=\rho_2(\lambda)$, and for any $[h]\in \Lambda\backslash N$, $\rho_1([h]^{-1})\circ \rho_2([h])$ is well defined and $\varkappa(\epsilon)$ $C^1$-close to the identity.
	For any $x\in \tilde U_1$, let $\tilde\Phi_2(x)$ be the center of mass of $h\mapsto \rho_1([h]^{-1})\circ\rho_2([h])$, i.e., the critical value of 
	$$y\to \int_{\Lambda\backslash N} d(y, \rho_1(h^{-1})\circ \rho_2(h)x)dh$$ 
	By \cite{GroveKarcher1973}, $\tilde \Phi_2:\tilde U_1\to \tilde U_1$ is well-defined diffeomorphism such that
	$$\tilde \Phi_2\circ\rho_2(h)(x)=\rho_1(h)\tilde \Phi_2(x),\quad \text{for any $h\in N$.}$$
	
	Secondly, by the construction, the quotient $\Phi_2:f_1^{-1}(B_{r/2}(p))\to f_1^{-1}(B_{r/2}(p))$ is $G$-equivariant and affine-equivariant.
	Moreover, for any $x\in B_{r/2}(p)$, the definition of $\Phi_2(x)$ does not depends on the choice of $p$.
	Thus $\Phi_2$ can be extended to a globally defined $G$-equivariant diffeomorphism such that 
	 $(\Phi_2)_*\mathfrak n_2=\mathfrak n_1$.
	 
\end{proof}
\begin{proof}[Proof of Theorem \ref{main-techthm}]
	\item \indent \indent 
	Let $\Psi:Y_1\to Y_2$, $\Phi_1:M\to M$ and $\Phi_2:M\to M$ to be diffeomorphism given by Proposition \ref{prop-bi-Lip-transformation}, Proposition \ref{prop-bundle-diffeo} and Proposition \ref{prop-affine-isom} respectively. 
	Let $\Phi=(\Phi_2\circ \Phi_1)^{-1}$, then $\Psi\circ f_1=f_2\circ \Phi$. 
	 Because $\Phi$ preserves the nilpotent Killing structures of $f_i$, it is an affine bundle isomorphism. By construction, $\Phi$ is $G$-equivariant, if in addition, $f_i$ are $G$-equivariant.
	 
\end{proof}

\begin{proof}[Proof of Theorem \ref{thm-lip-stable}]
	\item \indent\indent  
	Let $\mathfrak n_i$ be a nilpotent Killing structure on $M$ associated to $g_i$ ($i=1,2$). By Theorem \ref{thm-nilpotent-structure}, without loss of generality we assume that $g_i$ is $\mathfrak n_i$-invariant and $|\nabla^j R|\le A_j(n)$ ($j=0,1,\dots$).  By O'Neill formula, the orthonormal frame bundle $(FM,\tilde g_i)$, with a canonical metric induced by $g_i$, still admit a uniform two-sided sectional curvature bound.
	
	Since the differential action of $\mathfrak n_i$ on $FM$ is free and $\mathfrak n_i$ is pure, the quotient of $FM$ by the lifting nilpotent Killing structure 
	$\tilde{\mathfrak n}_i$ is still a Riemannian manifold $Y_i$. Moreover, since $\mathfrak n_i$ is pure and points all collapsing directions, the injectivity radius of $Y_i$ admits a uniform lower bound $i_0(n)$ (see \cite{Fukaya1988}, \cite{CFG1992}). 
	
	Thus, $\tilde{\mathfrak n}_i$ corresponds to an $O(n)$-equivariant affine bundle $\tilde f_i:(FM,\tilde g_i)\to Y_i$ satisfying (\ref{bounded-curv}), (\ref{Lip-equiv}.2), and (\ref{eps-collpase}.3).
	 
	If the dimension of $Y_i$ is the same, then by Theorem \ref{main-techthm}, $Y_1$ and $Y_2$ are diffeomorphic and $(FM,Y_i,\tilde f_i)$ are isomorphic as affine bundles. Consequently, the affine bundle isomorphism $\tilde \Phi:FM\to FM$ descends to a diffeomorphism $\Phi:M\to M$, such that $\Phi_*\mathfrak n_1=\mathfrak n_2$, and their infinitesimal action are conjugate by $\Phi$.
	
	Now let us assume that $n_1=\dim Y_1<\dim Y_2=n_2$. Since $g_1$ and $g_2$ are $L_0$-Lipschitz equivalent, the nilpotent Killing structure for $\tilde g_2$ can be chosen to be of the same dimension, i.e., there is another affine bundle $\tilde f'_2:(FM,\tilde g_2)\to (Y')^{n_1}$.
	
	By local compatibility of nilpotent structures (see \cite[Section 7]{CFG1992}), there is an affine bundle $\varphi:Y^{n_2}\to (Y')^{n_1}$ such that $\varphi \circ \tilde f_2$ is isomorphic to $\tilde f'_2$ as affine bundles. 
	Now by Theorem \ref{main-techthm}, $\tilde f'_2$ is conjugate to $\tilde f_1$, and thus after descending to $M$, we derive a diffeomorphism $\Phi:M\to M$ such that $\Phi_*\mathfrak{n}_2\subset \mathfrak{n}_1$ as a subsheaf. 
\end{proof}

\begin{proof}[Proof of Theorem \ref{thm-unique-nstr}]
	\item \indent \indent 
	It directly follows from Theorem \ref{thm-lip-stable}.
\end{proof}


\bibliographystyle{plain}
\bibliography{document}{}
	
\end{document}